\documentclass[a4paper,11pt]{article}
\usepackage{graphicx,import}
\usepackage{mathtools}
\usepackage{bbm}
\usepackage{url}
\usepackage{transparent}
\usepackage[rgb]{xcolor}
\usepackage{pdfpages}
\usepackage{framed}
\usepackage{mathrsfs}
\usepackage{dsfont}
\usepackage[utf8]{inputenc}
\usepackage[T1]{fontenc}
\usepackage[ngerman,main=english]{babel}
\usepackage{csquotes}
\usepackage{amsmath}
\usepackage{amscd}
\usepackage{amsthm}
\usepackage{amsfonts}
\usepackage{amssymb}
\usepackage{nicefrac}
\usepackage{comment}
\usepackage{marvosym}
\usepackage{amstext}
\usepackage[locale=DE]{siunitx}
\usepackage{graphicx}
\usepackage{float}
\usepackage{wrapfig}
\usepackage{epstopdf} 
\usepackage[a4paper,left= 30mm,right=30mm,top=30mm,bottom=30mm]{geometry}
\usepackage{stmaryrd}
\usepackage[shortlabels]{enumitem}
\setlist[enumerate,1]{(a)}
\usepackage{cancel}
\usepackage{hyperref}
\hypersetup{colorlinks,allcolors=black,final}
\usepackage{svg}
\usepackage{mwe}
\usepackage{multirow}
\usepackage{authblk}
\usepackage{subcaption}
\usepackage{caption}
\usepackage{accessibility}

\makeatletter %for linear system
\renewcommand*\env@matrix[1][*\c@MaxMatrixCols c]{%
  \hskip -\arraycolsep
  \let\@ifnextchar\new@ifnextchar
  \array{#1}}
\makeatother

\usepackage[ style=numeric, giveninits=true, eprint=false, url=false,isbn=false,  language=english, backend=biber
]{biblatex}
\renewbibmacro{in:}{% suppress "In:" prior to journal name
	\ifentrytype{article}
	{}
	{\printtext{\bibstring{in}\intitlepunct}}}
\AtEveryBibitem{\clearfield{month}}% don't print months
\AtEveryBibitem{\clearfield{day}}% don't print days

\newtheorem{theorem}{Theorem}[section]
\newtheorem*{theorem*}{Theorem}
\newtheorem{lemma}[theorem]{Lemma}

\theoremstyle{definition}
\newtheorem{definition}[theorem]{Definition}

\theoremstyle{remark}

\newtheorem*{example*}{Example}

\theoremstyle{plain}

\newcommand{\NN}{\mathbb{N}}

\newcommand{\PP}{\mathbb{P}}

\newcommand{\cH}{\mathcal{H}}

\newcommand{\indic}{\mathds{1}}
\providecommand{\xto}{\xrightarrow}
\providecommand{\E}{\mathbb E}

\newenvironment{nalign}{
    \begin{equation}
    \begin{aligned}
}{
    \end{aligned}
    \end{equation}
    \ignorespacesafterend
}

\begin{document}

\title{Stationary distribution of node2vec random walks on household models}
\author[1]{Lars Schroeder}
\author[1]{Clara Stegehuis}
\affil[1]{University of Twente - Faculty of Electrical Engineering, Mathematics and Computer Science}

\maketitle

\begin{abstract}
The node2vec random walk has proven to be a key tool in network embedding algorithms. These random walks are tuneable, and their transition probabilities depend on the previous visited node and on the triangles containing the current and the previously visited node.
Even though these walks are widely used in practice, most mathematical properties of node2vec walks are largely unexplored, including their stationary distribution. We study the node2vec random walk on community-structured household model graphs. We prove an explicit description of the stationary distribution of node2vec walks in terms of the walk parameters. We then show that by tuning the walk parameters, the stationary distribution can interpolate between uniform, size-biased, or the simple random walk stationary distributions, demonstrating the wide range of possible walks. We further explore these effects on some specific graph settings.
\end{abstract}

\begin{keywords}
    node2vec, random walks, stationary distribution, household models
\end{keywords}

\section{Introduction}
Random walks on graphs have captivated researchers for decades. Beyond their theoretical appeal, random walks serve as a fundamental technique in developing algorithms that extract insights from network data. Examples include community detection, node ranking, dimension reduction, and sampling~\cite{brinkman2007vertex,bohlin2014community,ribeiro2010estimating}. Traditionally, many studies have focused on simple random walks, where the walker transitions to a uniformly chosen neighbor at each step. However, several more complex types of random walks have proven to be useful in algorithms related to community detection and learning on networks~\cite{grover_node2vec_2016,alon2007non}.

For example, the non-backtracking random walk, where nodes are not allowed to walk on the edge they followed in the last step, has proven fundamental to community detection on the stochastic block model~\cite{krzakala2013}. Unlike simple random walks, which induce a Markov chain on the nodes, such random walks form a second-order stochastic process, as they depend on one past state. While second-order processes are more complex to analyze, such walks mix faster~\cite{alon2007non}, and have shown better concentration properties in spectral methods~\cite{Kawamoto_2016}. 

A second, more recently proposed second-order random walk is the node2vec random walk~\cite{grover_node2vec_2016}, which contains one parameter that controls the non-backtracking properties of the walker, and a second parameter that controls the likelihood of exploring nodes further from the current node as compared to staying in a close neighborhood. This random walk has been a key ingredient in learning network embeddings~\cite{grover_node2vec_2016}, and has been applied in many problems of wide ranges, from finding pathways in biological networks~\cite{kim2018relation} to detecting urban functional regions~\cite{cai2022discovery} or representing the content of scientific papers~\cite{KAZEMI2020101794}. 

Although experimental results have shown promising results on the application of node2vec random walks, much less is known about its theoretical behavior. In particular, results about its stationary distribution only contain results where the second parameter (controlling local vs global exploration) does not play a role \cite{meng_analysis_2020,barot_community_2021}, and the random walk essentially reduces to a mixture between a simple random walk and a non-backtracking random walk, which has the same stationary distribution as a standard random walk. However, the main reason behind the success of node2vec in practice is the tuneability of local and global exploration. The effect of this parameter on the concentration of the stationary distribution is completely unknown. 

In this paper, we therefore analyse the stationary distribution of the node2vec random walk with general parameters on a popular graph model that contains communities: the household model~\cite{Ball_Sirl_2012}.

We provide an explicit formula for the stationary distribution on this model and show it strongly depends on the chosen parameters of the random walks through analysis of specific cases. While there are parameter regimes of the node2vec random walks in which the stationary distribution is close to the one of a simple random walk, it is possible to tune the parameters such that larger and more strongly connected communities are favored or disfavored.

\subsection{Notation}
Through the rest of this paper, we will assume that $G=(V,E)$ is a finite graph with vertex set $V$ and edge set $E$. We will write $E$ for the undirected edge set and $\bar{E}$ for the corresponding directed edge set, i.e. the set where each edge in $E$ gets replaced by two directed edges that have opposite directions.
A graph is connected if there exists a path between $u$ and $v$ for all $u,v \in V$ and it is aperiodic if the greatest common divisor of the length of its cycles is one. We call a graph simple if it does not contain multi-edges and self-loops. 
Further, we say that $X_n = o_{\PP}(1)$ for a sequence of random variables $X_1,X_2,\ldots$ if $\lim_{n \to \infty} \PP\left(\left|{X_n}\right| \geq \varepsilon\right) = 0$ for all $\varepsilon > 0$. Finally, we denote convergence in probability by $X_n\xto[]{\PP}Y$ and almost sure convergence by $X_n\xto[]{a.s.}Y$ of the sequence of random variables $X_1,X_2,\ldots$ to the random variable $Y$.

\subsection{Organization of the paper}
In Section~\ref{sec:node2vecrw}, we introduce node2vec random walks and previous results. In Section~\ref{sec:mainresultsonhousehold}, we define the underlying graph models and state our main result. In Section~\ref{sec:examplesandapplications}, we study specific examples and applications to get a better intuition of the results. Section~\ref{sec:proofofmainthm} contains the proof of the main result.

\section{node2vec random walks}\label{sec:node2vecrw}
We start by introducing node2vec random walks, as introduced in~\cite{grover_node2vec_2016}. 
\begin{definition}(node2vec random walk) \label{def:node2vec}\\
    Let $\alpha,\beta,\gamma \ge 0$. A random walk $(X_i)_{i\in \NN}$ with state space $V$ is called a node2vec random walk if for all $\{v,w\} \in E$
    \begin{align}\label{eq:transprobn2v}
        \PP(X_{i+1}=w|X_{i}=v,X_{i-1}=u) \propto \begin{cases}
		\alpha ~\text{ if } u=w,\\
        \beta ~\text{ if } \{u,w\}\in E~\text{ and } u\neq w,\\
        \gamma ~\text{ if } \{u,w\}\notin E ~\text{ and } u\neq w,\\
	  0 ~\text{ otherwise},
	\end{cases}
    \end{align}
    for some $\alpha,\beta,\gamma\geq 0$.
\end{definition}

We initialize the node2vec random walk by uniformly choosing a directed edge in $\bar{E}$. 
Our goal is to study the stationary distribution of node2vec random walks. However, the transition probabilities depend on the current and the previous position (see Figure~\ref{fig:node2vecdef}), and therefore the node2vec random walk is not a Markov chain on $V$. However, if we change the state space to the space of the directed edges $\bar{E}$, then the transition probabilities only depend on the current directed edge $(u,v)$ instead of the previous and current vertex. Therefore, the node2vec random walk forms a Markov chain on $\bar{E}$, making it a so-called second-order random walk. Note that if $\alpha,\beta,\gamma >0$ and the graph is connected, the node2vec random walk is recurrent, both on $E$ and $\bar{E}$. If the graph is connected, simple, finite and aperiodic, the node2vec random walk on $\bar{E}$ admits a unique stationary distribution $\bar{\pi}$ that does not depend on the choice of the starting position. In this setting, $\bar{\pi}$ is a stationary distribution on the directed edges. By denoting the target node of a directed edge $e$ with $e(1)$, we can then define
\begin{equation}\label{eq:projectionstatdistredgestonodes}
    \pi(v) = \sum_{e\in\bar{E}: e(1)=v}\bar{\pi}(e),
\end{equation}
which is the sum of all the stationary probabilities on edges that lead to vertex $v$. We call these projected probabilities the stationary distribution of the node2vec random walk. More details on this approach can be found in~\cite{meng_analysis_2020}. Throughout the paper, we will typically write $\pi(v)$ for the (possibly projected) stationary probability of a random walk of a node $v$, while we refer to the stationary probabilities on the edges by $\bar{\pi}$. When a possible confusion could occur, we specify the random walk with an index, e.g. $\pi^{SRW}$ for the simple random walk. 

\begin{figure}[tbp]
    \centering
    \includegraphics[width=0.4\linewidth]{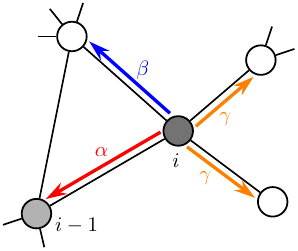}
    \caption{Illustration of the transition rates of a node2vec random walk.}
    \alt{Illustration of the transition rates of a node2vec random walk.}
    \label{fig:node2vecdef}
\end{figure}

The parameters $\alpha,\beta$ and $\gamma$ play a crucial role in the behavior of the random walk. Increasing $\alpha$ leads to more backtracking, while small values of $\alpha$ make it unlikely that a walker will revisit the previously visited edge. Increasing $\beta$ increases the tendency for the random walker to stay in a neighborhood with many triangles, such as cliques. On the other hand, increasing $\gamma$ makes it more likely that the next chosen vertex is not connected to the previously visited vertex. Intuitively, this makes it easier for the walker to explore vertices that are further in graph distance from the starting node. 
Therefore, intuitively, a node2vec random walk with large $\beta$ is more similar to a breadth-first search while with large $\gamma$, the walk is more similar to a depth-first search. Special cases of the node2vec random walk include the simple random walk: $\alpha = \beta = \gamma$ and the non-backtracking random walk: $\alpha = 0, \beta = \gamma$. Apart from these special cases and the case $\alpha >0, \beta = \gamma$, theoretical properties of node2vec random walks have not been studied. In particular, the stationary distribution for general parameters is not known in closed form.

In the special case $\alpha > 0,\beta = \gamma$, the stationary distribution of the node2vec random walk coincides with the stationary distribution of the simple random walk~\cite{meng_analysis_2020}. However, this case only explores the non-backtracking parameter of the node2vec random walks, and the parameters $\beta$ and $\gamma$ that tune the local and global exploration of the graph remain unchanged. 
This result can be formalized as
\begin{theorem}[Theorem 3.1. of~\cite{meng_analysis_2020}]\label{thm:statdistr-slightbacktracking}
    Let $G$ be a connected, finite, simple, unweighted, undirected, aperiodic graph. Then, the stationary distribution for a node2vec random walk with $\beta = \gamma$ and $\alpha > 0$ is given by
    \begin{align}
        \pi(v) = \frac{d_v}{2|E|}
    \end{align}
     for $v\in V$.
\end{theorem}

Inspired by this Theorem, we will prove Theorem \ref{thm:statdistrofystar} in Section~\ref{sec:proofofmainthm}. 

\section{Main results on household graphs}\label{sec:mainresultsonhousehold}

In this paper, we study the stationary distribution of the node2vec random walk on so-called household model~\cite{Ball_Sirl_2012}, a natural model for networks with community structures. The model consists of multiple communities that are densely connected inside, the households, while connections between communities are more sparse. Household models have been used in various settings, most notably in epidemic spreading models~\cite{BALL199941,BALL201563,HOUSE200829,stegehuis2016epidemic,hofstad2016,coupechoux_how_2014}. 

We now present the specific household model that we study. 
For that purpose, we start by defining automorphic communities. Recall that an automorphism $\varphi$ on a graph $H=(V,E)$ with adjacency matrix $A$ is a bijection $\varphi:V\to V$ that fulfills $A_{ij}=A_{\varphi(i)\varphi(j)}$ for all $i,j\in \{1,\ldots,\lvert V\rvert\}$. 
Two nodes $v,v'\in V$ are called automorphically equivalent if there exists an automorphism that maps one to the other, i.e. $\varphi(v)=v'$. Then, an automorphic community is defined in the following way:

\begin{definition}(Automorphic community)
    An automorphic community of size $k$ is a connected subgraph with $k$ nodes in which all pairs of nodes are automorphically equivalent.
\end{definition}

Examples of automorphic communities include cliques, rings and cycles. 
The construction of the household model begins with a given graph, the so-called universe graph $G'=(V',E')$, which models the connections between households. We then obtain the household model $G$ by replacing each node $v' \in V'$ with an automorphic community of size $d_{v'}$. More precisely,

\begin{definition}(Household model)\label{def:simplehousehold}
    Let $G'$ be a connected, finite, simple, unweighted, undirected, aperiodic graph with $\lvert V' \rvert = n$. Label the nodes in $G'$ by $(v'_1,\ldots,v'_n)$ and let $(d_1,\ldots,d_n)$ be the corresponding degree sequence. Let $H_1,\ldots,H_n$ be graphs such that $H_i$ is an automorphic community of size $d_i$ for all $i\in \{ 1,\ldots,n\}$.
    Then, the household model $G$ is created by performing the following two steps for all $i\in \{ 1,\ldots,n\}$
    \begin{itemize}
        \item Replace $v'_i \in V'$ by $H_i$.
        \item Add an edge from each node of $H_i$ to exactly one of the neighbors of $v'_i$.
    \end{itemize}
    The resulting graph is the household model $G$ with the corresponding universe graph $G'$.
\end{definition}

Figure~\ref{fig:householdtrafo} shows an example of the household model transformation. We want to study the stationary distribution on household models, so we require them to be aperiodic. There are cases in which the universe graph is aperiodic but the corresponding household model is not, but in these cases, the household model does not contain any triangles. Since the distinction between moving on triangles and non-triangles is the main feature of node2vec random walks, in the rest of this paper we assume that the household models contain at least one triangle and thus are aperiodic. We also assume that $\alpha,\beta,\gamma > 0$ to avoid issues with the aperiodicity of the random walk.

\begin{figure}[tbp]
    \centering
    \includegraphics[width=1\linewidth]{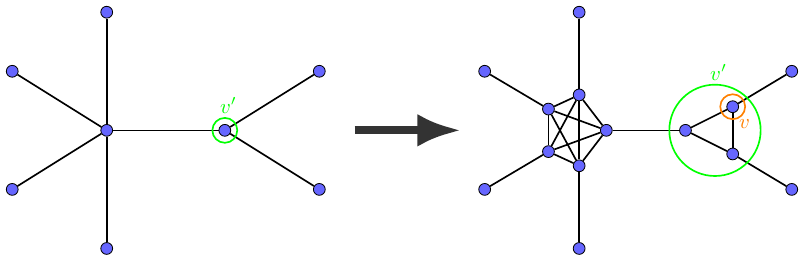}
    \caption{Example of the household model transformation. Left: the universe graph $G'$. A vertex $v'\in G'$ of degree 3 corresponds to a clique with 3 nodes $(C_3)$ in the household graph $G$ (right). In this example, $type(v')=C_3$.}
    \alt{Illustration of the household model transformation. A node of degree 3 gets replaced by a 3-clique and a node of degree 5 gets replaced by a 5-clique.}
    \label{fig:householdtrafo}
\end{figure}

To state our main theorem, we define the following notation. Let $\cH_G$ be the set of unique communities of $\{H_1,\cdots,H_n\}$ in $G$. Define $n_H^G$ as the number of communities in $G$ of type $H \in \cH_G$ and let $|H|$ be the number of vertices in community $H$. For $v'\in V'$ denote by $type(v')$ the graph $H \in \cH_G$ that describes its corresponding community.
We define the random variable $\tau^{(H)}$ as the number of jumps the node2vec random walk makes within a community of type $H \in \cH_G$ before leaving it. 
In this random variable, we include the first jump to enter $H$, so that the time spent in a community of size $1$ is always one.

We now state our main result which fully characterizes the stationary distribution of a node2vec random walk on a household model. 

\begin{theorem}\label{thm:mainthm}
    Let $\alpha,\beta, \gamma > 0$ and let $v \in V$ be in a community of type $\Tilde{H} \in \cH_G$ in the household model $G$. Then, 
    \begin{align}
        \pi(v) = \frac{\E[\tau^{(\Tilde{H})}]}{\sum\limits_{H\in \cH_G } \lvert H \rvert n_H^{G} \E[\tau^{(H)}]}.
    \end{align}
\end{theorem}
Thus, the time that the node2vec random walk spends in each community of type $H$ is dictated by the random variable $\tau^{(H)}$, whose behavior is strongly affected by the parameters of the random walk. We will show in Section~\ref{sec:examplesandapplications} that this random variable can be computed explicitly in some natural cases, so that we can derive the explicit influence of $\alpha, \beta$ and $\gamma$. Section~\ref{sec:proofofmainthm} contains the proof of the theorem.

\paragraph{Relation to previous work.}
As mentioned previously, most results on node2vec random walks only apply to simple parameter settings, such as Theorem~\ref{thm:statdistr-slightbacktracking}. In contrast, Theorem~\ref{thm:mainthm} applies to all parameter settings. In addition to that, \cite{meng_analysis_2020} studies the relaxation time and the coalescence time of node2vec random walks with general parameters on specific small networks numerically, suggesting that these times are usually smallest when $\alpha$ and $\beta$ are small.

Other analytical analyses of node2vec include \cite{barot_community_2021}, which uses node2vec random walks to perform community detection on a stochastic block model. Although the authors show that node2vec outperforms the simple random walk using community detection, the authors again restrict to the setting $\alpha > 0,\beta = \gamma$, in which only the non-backtracking parameter $\alpha$ plays a role and where there is no distinction between $\beta$ and $\gamma$. However, specifically in models with a community structure and therefore many cliques, $\beta$ and $\gamma$ are expected to affect the stationary distribution significantly.

Apart from papers on node2vec random walks specifically, there are results on general second-order random walks, which include node2vec random walks. In \cite{fasino_hitting_2023} various results on hitting and return times are presented. The authors also link the second-order mean return time to the stationary distribution but as an example they only provide an explicit formula for random walks whose stationary distribution coincides with the stationary distribution of the simple random walk.

Overall, only few papers contain mathematical results on node2vec random walks and the node2vec algorithm, which is surprising given the popularity of the node2vec algorithm in computer science. Specifically, the contribution of the essential triangle walk parameter $\beta$ has been virtually always overlooked until now.

\section{Examples and Applications}\label{sec:examplesandapplications}
We now present several examples of specific cases of Theorem~\ref{thm:mainthm}, where we can compute $\mathbb{E}[\tau^{(H)}]$ explicitly to investigate the influence of the parameters $\alpha,\beta,\gamma$. 

\subsection{Average time in cliques}
The most natural example of community types are cliques, as in the original household model~\cite{Ball_Sirl_2012}. 
A $k$-clique is a complete subgraph of $V$ on $k$ vertices and a maximal clique $C_k$ is a $k$-clique that is not a subset of a bigger clique. This means that for example two vertices in a maximal $3$-clique form a $2$-clique but not a maximal $2$-clique. In the following, we will only consider maximal cliques. The next theorem provides the expected time spent in a $k$-clique which, combined with Theorem~\ref{thm:mainthm}, yields an explicit stationary probability for node2vec random walks on clique household models.

\begin{theorem}\label{thm:exptime-cliques} Let $H_1,\dots, H_n = C_{k_1},\dots C_{k_n}$, where $k_i = |H_i|$. For all $k \in \NN$ it holds
    \begin{align}\label{eq:evaluecliques}
        \E[\tau^{(C_k)}] = \frac{\alpha+(k-1)(\alpha+(k-2)\beta+2\gamma)}{\alpha+(k-1)\gamma}.
    \end{align}
\end{theorem}

\begin{proof}
    First, we calculate the probability of remaining within the $k$-clique for a specific number of steps. The minimum number of steps in a clique is 1, which occurs if the random walk exits the clique immediately after entry. Upon the initial jump to a clique node, there are two possible actions. According to the transition rates defined in node2vec random walks, as given in~\eqref{eq:transprobn2v}, we may either return to the originating node with a probability proportional to $\alpha$, or select one of the remaining $k-1$ nodes in the clique, each with a probability proportional to $\gamma$. For the random walk to exit the clique immediately after entry, it must return to the originating node. Consequently, this results in:
    \begin{align}\label{eq:cliquestayonestep}
        \PP(\tau^{(C_k)}=1) = \frac{\alpha}{\alpha+(k-1)\gamma}.
    \end{align}
    To remain in the clique for exactly two steps, the random walk must take one step within the clique after entering and then exit in the subsequent step. This involves selecting one of the other $k-1$ nodes in the clique for the internal step, followed by a jump to a neighboring node outside the clique. Consequently, we obtain
    \begin{align}
        \PP(\tau^{(C_k)}=2) = \frac{(k-1)\gamma}{\alpha+(k-1)\gamma} \cdot \frac{\gamma}{\alpha+(k-2)\beta+\gamma}.
    \end{align}
    To remain for exactly three steps, we need to add an extra step inside the clique. The probability to revisit the previous node is proportional to $\alpha$ and the probability to go to one of the other $k-2$ nodes is proportional to $\beta$, so that we obtain
    \begin{align}
        \PP(\tau^{(C_k)}=3) &= \frac{(k-1)\gamma}{\alpha+(k-1)\gamma} \cdot \frac{\alpha+(k-2)\beta}{\alpha+(k-2)\beta+\gamma} \cdot \frac{\gamma}{\alpha+(k-2)\beta+\gamma}.
    \end{align}
    Higher amounts of steps follow similarly by adding more steps inside the clique. Because of the symmetry of cliques, for $l\ge 2$, we obtain 
    \begin{align}  
        \PP(\tau^{(C_k)}=l) = \frac{(k-1)\gamma}{\alpha+(k-1)\gamma} \cdot \left(\frac{\alpha+(k-2)\beta}{\alpha+(k-2)\beta+\gamma}\right)^{l-2} \cdot \frac{\gamma}{\alpha+(k-2)\beta+\gamma}.
    \end{align}
    We now want to calculate the expected value
    \begin{align}
        \E[\tau^{(C_k)}] = \sum\limits_{l=1}^{\infty} l \cdot \PP(\tau^{(C_k)}=l).
    \end{align}
    For that we define
    \begin{align}
        q \coloneqq \frac{\alpha+(k-2)\beta}{\alpha+(k-2)\beta+\gamma}.
    \end{align}
    We prepare the calculation of the expected value by computing 
    \begin{align}
        \sum\limits_{l=2}^{\infty} l \cdot \left(\frac{\alpha+(k-2)\beta}{\alpha+(k-2)\beta+\gamma}\right)^{l-2} = \sum\limits_{l=2}^{\infty} l \cdot q^{l-2} = \sum\limits_{l=1}^{\infty} (l+1)\cdot q^{l-1}.
    \end{align}
    Using the geometric series we obtain that
    \begin{nalign}
        \sum\limits_{l=2}^{\infty} l \cdot q^{l-2} &= \frac{1}{(1-q)^2} + \frac{1}{1-q} = \frac{(\alpha+(k-2)\beta+\gamma)^2}{\gamma^2} + \frac{\alpha+(k-2)\beta+\gamma}{\gamma}\\
        &=\frac{(\alpha+(k-2)\beta+\gamma)^2+\gamma (\alpha+(k-2)\beta+\gamma)}{\gamma^2}\\
        &=(\alpha+(k-2)\beta+\gamma)\frac{\alpha+(k-2)\beta+2\gamma}{\gamma^2}.
    \end{nalign}
    Next, we compute
    \begin{nalign}
        \sum\limits_{l=2}^{\infty} l \cdot\PP(\tau^{(C_k)}=l) 
        &= \frac{(k-1)\gamma}{\alpha+(k-1)\gamma}  \frac{\gamma}{\alpha+(k-2)\beta+\gamma} \sum\limits_{l=2}^{\infty} l \cdot q^{l-2}\\
        &= \frac{(k-1)\gamma^2}{\alpha+(k-1)\gamma}  \frac{\alpha+(k-2)\beta+\gamma}{\alpha+(k-2)\beta+\gamma} \frac{ \alpha+(k-2)\beta+2\gamma}{\gamma^2} \\
        &= \frac{(k-1)(\alpha+(k-2)\beta+2\gamma)}{\alpha+(k-1)\gamma} .
    \end{nalign}
    Adding the summand~\eqref{eq:cliquestayonestep} yields the expected value
    \begin{nalign}
        \E[\tau^{(C_k)}] &= \sum\limits_{l=1}^{\infty} l \cdot \PP(\tau^{(C_k)}=l) = \frac{\alpha}{\alpha+(k-1)\gamma} +\sum\limits_{l=2}^{\infty} l \cdot \PP(\tau^{(C_k)}=l) \\
        &= \frac{\alpha}{\alpha+(k-1)\gamma} + \frac{(k-1)(\alpha+(k-2)\beta+2\gamma)}{\alpha+(k-1)\gamma}\\
        &= \frac{\alpha + (k-1)(\alpha+(k-2)\beta+2\gamma)}{\alpha+(k-1)\gamma}.
    \end{nalign}
\end{proof}

\subsection{Average time in ring communities}

We now turn to ring communities, as illustrated in Figure~\ref{fig:ringcommunity}.
\begin{figure}[tbp]
    \centering
    \includegraphics[width=0.4\linewidth]{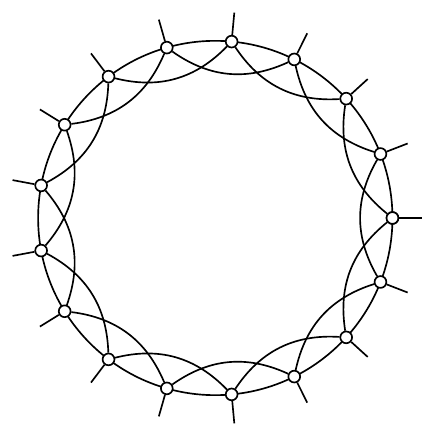}
    \caption{Illustration of a ring community of size $k=18$.}
    \alt{Illustration of a ring community of size 18. The nodes are depicted in a ring and every node is connected to its two nearest neighbors and has one outside arm.}
    \label{fig:ringcommunity}
\end{figure}
In these communities, the $k$ vertices are aligned in a circle and each vertex is connected to its four nearest neighbors.
Such ring communities have received specific attention in~\cite{meng_analysis_2020}, which studied the relaxation times of node2vec random walks with different parameter choices on ring communities.

For community sizes $k \le 5$, ring communities coincide with $k$-cliques. We therefore now focus on consider ring communities of size $k \ge 6$. For ring communities of size $k \ge 7$, the transition probabilities depend on whether the previous step was to a neighbor on the ring or whether it was a long step through the inside of the ring. Indeed, after a step to a neighbor on the ring, the previous and the current node are part of two triangles. On the other hand, after a long step on the inside, the previously visited node and the current node are only part of one triangle, and therefore the transition probabilities are different. For $k=6$, the ring is small, and the previously visited and the current node are always part of two triangles, so that the above phenomenon does not occur in this case. 

Since the node degrees do not depend on $k$, the stationary probabilities for all nodes in ring communities of size $k \ge 7$ are equal in the entire household model. For example, a node in a ring community of size seven and a node in a ring community of size 13 have the same stationary probability. Denote a ring community of size $k$ by $R_k$. We now compute $\mathbb{E}[\tau^{(R_k)}]$. Since the calculations are of the same type as for the cliques, we omit most details.

\underline{$k = 6$} \\
Using Definition~\ref{def:node2vec} we calculate the following probabilities to stay a certain number of steps in the community.
\begin{nalign}
    \PP(\tau^{(R_6)} = 1) =&~ \frac{\alpha}{\alpha + 4\gamma} \\
    \PP(\tau^{(R_6)} = 2) =&~ \frac{4\gamma}{\alpha+4\gamma} \cdot \frac{\gamma}{\alpha + 2\beta + 2\gamma} \\
    \PP(\tau^{(R_6)} = 3) =&~ \frac{4\gamma}{\alpha+4\gamma} \cdot \frac{\alpha + 2\beta + \gamma}{\alpha +2\beta +2\gamma} \cdot \frac{\gamma}{\alpha + 2\beta + 2\gamma} 
\end{nalign}
Because of the symmetry of this ring, for $l\ge 2$ we derive
\begin{align}  
    \PP(\tau^{(R_6)}=l) = \frac{4\gamma}{\alpha+4\gamma} \cdot \left(\frac{\alpha + 2\beta + \gamma}{\alpha +2\beta +2\gamma}\right)^{l-2} \cdot \frac{\gamma}{\alpha + 2\beta + 2\gamma}
\end{align}
which leads to the expected value
\begin{align}
    \E[\tau^{(R_6)}] = \frac{5\alpha+8\beta+12\gamma}{\alpha+4\gamma}.
\end{align}

\underline{$k \ge 7$} \\
As mentioned above, the transition probabilities for $k \ge 7$ depend on the type of the previous step. After a short step on the outside of the ring, the previous and current node are part of two triangles and the probability of leaving the community after a short step $p_s$ is given by
\begin{align}
    p_s &\coloneqq \frac{\gamma}{\alpha + 2\beta + 2\gamma}
\end{align}
and after long step inside the ring, the previous and current node are part of only one triangle and the probability of leaving the community after a long step $p_l$ is given by
\begin{align}    
    p_l &\coloneqq \frac{\gamma}{\alpha + \beta + 3\gamma}.
\end{align}
The probability to leave after exactly one step is the same as for $k=6$ with 
\begin{align}
    \PP(\tau^{(R_k)} = 1) = \frac{\alpha}{\alpha + 4\gamma}.
\end{align}
To stay exactly two steps in the community, we need to differentiate based on whether the step inside the community is short or long. After entering the community, the options for the next step are two short steps, two long steps, or leaving the community. 
This means, either we choose a short step with probability $\frac{2\gamma}{\alpha+4\gamma}$ and then leave with probability $p_s$ or a long step with probability $\frac{2\gamma}{\alpha+4\gamma}$ and then leave with probability $p_l$.
We obtain
\begin{align}
    \PP(\tau^{(R_k)} = 2) &= \frac{2\gamma}{\alpha+4\gamma}p_s + \frac{2\gamma}{\alpha+4\gamma} p_l\nonumber\\
    &= \frac{2\gamma}{\alpha+4\gamma}\frac{\gamma}{\alpha + 2\beta + 2\gamma} + \frac{2\gamma}{\alpha+4\gamma} \frac{\gamma}{\alpha + \beta + 3\gamma}.
\end{align}
For larger step values we need to consider all possible combinations of short and long steps that were made inside the community. For $k=3$, we have exactly two steps inside the community, leading to the combinations short-short, short-long, long-short and long-long, or more formally
\begin{nalign}
    \PP(\tau^{(R_k)} = 3) =&~ \frac{2\gamma}{\alpha+4\gamma} \frac{\alpha +\beta}{\alpha+2\beta+2\gamma} \frac{\gamma}{\alpha + 2\beta + 2\gamma} + \frac{2\gamma}{\alpha+4\gamma} \frac{\beta+\gamma}{\alpha+2\beta+2\gamma} \frac{\gamma}{\alpha + \beta + 3\gamma} \nonumber\\
    &+\frac{2\gamma}{\alpha+4\gamma} \frac{\beta+\gamma}{\alpha+\beta+3\gamma} \frac{\gamma}{\alpha + 2\beta + 2\gamma} + \frac{2\gamma}{\alpha+4\gamma} \frac{\alpha+\gamma}{\alpha+\beta+3\gamma} \frac{\gamma}{\alpha + \beta + 3\gamma}.
\end{nalign}
Formalizing this for $m$ steps, we obtain
\begin{align}
    \PP(\tau^{(R_k)} = m) = \frac{2\gamma}{\alpha+4\gamma} \sum\limits_{(x_1,\ldots,x_{m-1})\in \{ s,l \}^{m-1}}\left( \prod\limits_{j=1}^{m-2} q_{x_j,x_{j+1}} \right) p_{x_{m-1}}
\end{align}
where
\begin{align}
    q_{i,j} \coloneqq 
    \begin{cases}
        \frac{\alpha +\beta}{\alpha+2\beta+2\gamma}, &\text{ if } i=j=s,\\
        \frac{\beta+\gamma}{\alpha+2\beta+2\gamma}, &\text{ if } i=s,j=l,\\
        \frac{\beta+\gamma}{\alpha+\beta+3\gamma}, &\text{ if } i=l,j=s,\\
        \frac{\alpha+\gamma}{\alpha+\beta+3\gamma}, &\text{ if } i=j=l.
    \end{cases}
\end{align}

This expression does not seem to admit a general closed form that would be helpful for calculating the expectation. This shows that it is not easy to obtain an explicit expression of the expected time spent in a community (and with that the stationary distribution), even for simple community types like ring communities.

\subsection{Clique household model with Poissonian clique sizes}\label{subsec:Clique household model with Poissonian clique sizes}

In this section, we test the effects of the parameters on the stationary distribution on a clique household model. We sampled the degrees of $100$ nodes from a $Poi(4)$-distribution and formed the universe graph $G'$ by using a configuration model \cite{bollobas_random_2008}. We created the household model $G$ according to Definition~\ref{def:simplehousehold} with cliques as community type. For clique communities, we can use Theorem~\ref{thm:mainthm} and Theorem~\ref{thm:exptime-cliques} to calculate the stationary distribution of the nodes explicitly. Note that the stationary probability of a node only depends on its degree in the household model because each node of degree $k$ is part of a maximal $k$-clique. By Theorem \ref{thm:mainthm}, all nodes within a clique have the same projected stationary probability. 

Figure~\ref{fig:plotscliquespoissondegreesbeta10} plots the stationary probability on the $y$-axis against the degrees of the nodes in the household model $G$ on the $x$-axis. The orange bars represent the stationary probabilities of a node2vec random walk on that network with the stated parameters $\alpha, \beta, \gamma$ while the blue bars show the stationary probabilities of the simple random walk for comparison. Since the transition probabilities are only determined by the ratio of $\alpha, \beta$ and $\gamma$, we fix $\gamma = 1$ and test for different values of $\alpha$ and $\beta$. For $\beta = \gamma$, the stationary distributions of the node2vec random walk and the simple random walk coincide (see Theorem~\ref{thm:statdistr-slightbacktracking}), so we study the cases $\beta = 0.1$ and $\beta = 10$ instead.
    \begin{figure}[tbp]
        \centering
        \begin{minipage}{0.3\textwidth}
            \centering
            \includegraphics[scale=0.55]{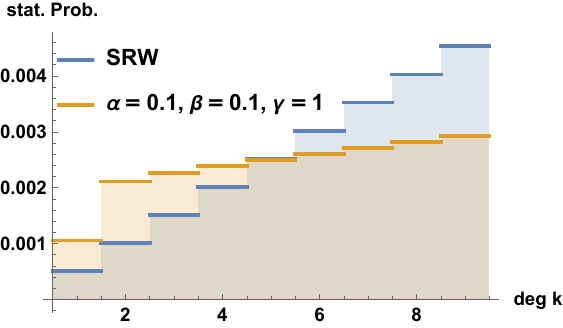}
        \end{minipage}\hfill
        \begin{minipage}{0.3\textwidth}
            \centering
            \includegraphics[scale=0.55]{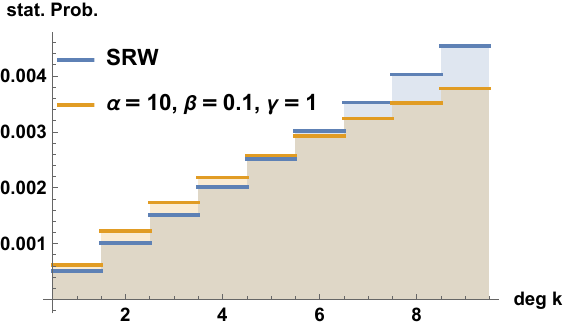}
        \end{minipage}\hfill
        \begin{minipage}{0.3\textwidth}
            \centering
            \includegraphics[scale=0.55]{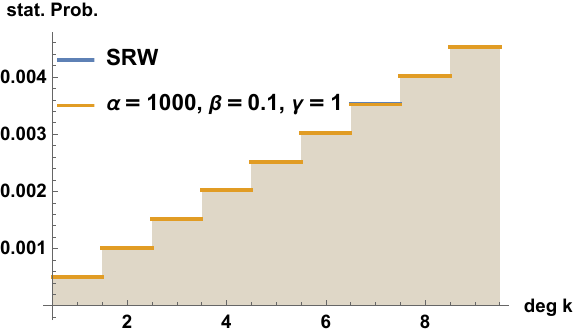}
        \end{minipage} 
        \begin{minipage}{0.3\textwidth}
            \centering
            \includegraphics[scale=0.55]{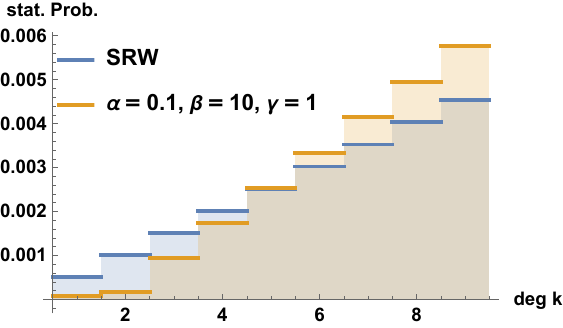}
        \end{minipage}\hfill
        \begin{minipage}{0.3\textwidth}
            \centering
            \includegraphics[scale=0.55]{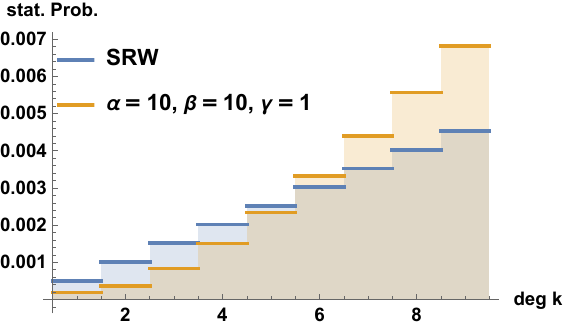}
        \end{minipage}\hfill
        \begin{minipage}{0.3\textwidth}
            \centering
            \includegraphics[scale=0.55]{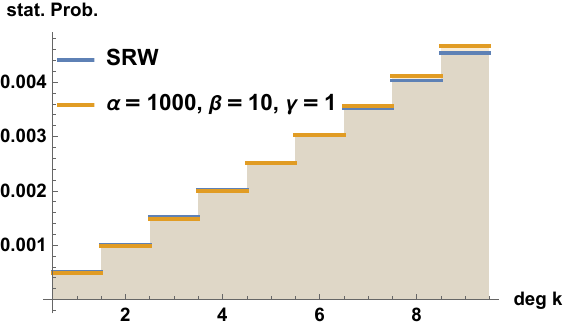}
        \end{minipage} 
        \caption{Stationary distribution of the node2vec random walk and the simple random walk for different values of $\alpha$ and $\beta$ on a large household model with Poissonian clique sizes.}
        \alt{Illustration of 6 bar plots that depict stationary distributions of node2vec random walks for different parameters and of the simple random walk for comparison.}
        \label{fig:plotscliquespoissondegreesbeta10}
    \end{figure}

The figures suggest that increasing the value of $\beta$ leads to more visits of high-degree nodes because the node2vec random walk remains inside the cliques for more time steps. The parameter $\alpha$ does not seem to have a clear effect. For large $\alpha$, the stationary distributions of the node2vec random walk seems to converge to the stationary distribution of the simple random walk. We explore this in more detail in the next section.

\subsection{Limit cases}

Although Theorem~\ref{thm:mainthm} is explicit, it does not give an intuitive result on the influence of the parameters. We therefore investigate the stationary distributions when the parameters $\alpha,\beta$ or $\gamma$ tend to zero or infinity. We specifically focus on clique communities where the sizes are Poisson distributed for specific limit cases. 

Let $S$ be the size of a clique in the household model, let $n$ be the number of total cliques and $n_{C_k}^G$ the number of $k$-cliques in the household model $G$. We assume that $S \sim Poi(\lambda)$. Then, by the law of large numbers,
\begin{align}
    n_{C_k}^G = \PP(S=k)\cdot n (1+o_{\PP}(1)).
\end{align} 
For $k=1$ we have
\begin{align}
    \E[\tau^{(C_1)}] = 1
\end{align}
for any choice of parameters, so for the rest of this section we assume that $k\ge 2$. 
Applying the limit $\alpha \to \infty$ to formula~\eqref{eq:evaluecliques} for the expected number of steps for $k$-cliques, we obtain
\begin{align}
    \E[\tau^{(C_k)}] = \frac{\alpha+(k-1)(\alpha+(k-2)\beta+2\gamma)}{\alpha+(k-1)\gamma} \xto[\alpha \to \infty]{} k.
\end{align}
Plugging this expectation into Theorem~\ref{thm:mainthm} yields
\begin{align}
  \pi(v) = \frac{l}{\sum_{k=1}^n k^2 n_{C_k}^G} = \frac{l}{n \E[S^2]}(1+o_{\PP}(1)) = \frac{l}{n(\lambda^2+\lambda)}(1+o_{\PP}(1))
\end{align}
for any node $v$ which is part of a $l$-clique. Since each clique of size $S_i$ contains $\binom{S_i}{2}$ edges inside and $\frac{S_i}{2}$ outside arms to other cliques, the total number of edges within the graph can be written as
\begin{align}
    |E|= \sum\limits_{i=1}^n \left(\binom{S_i}{2}+\frac{S_i}{2}\right) =\sum\limits_{i=1}^n \left(\frac{1}{2}(S_i(S_i-1))+\frac{S_i}{2}\right) = \frac{1}{2}\sum\limits_{i=1}^n S_i^2.
\end{align}
This implies that under the standard random walk, the stationary distribution of any node $v$ which is part of a $l$-clique can be written as
\begin{align}
    \pi^{SRW}(v) = \frac{d_v}{2 |E|} = \frac{l}{\sum_{i=1}^n S_i^2} = \frac{l}{n \E[S^2]}(1+o_{\PP}(1)) = \frac{l}{n(\lambda^2+\lambda)}(1+o_{\PP}(1)),
\end{align}
by using the law of large numbers. Thus, when $\alpha\to\infty$, the stationary probabilities of the node2vec random walk and the simple walk are asymptotically equal. 

To visualize the limit cases, we used the same household model as in Section~\ref{subsec:Clique household model with Poissonian clique sizes} (i.e. $n=100$, $S\sim Poi(4)$). Figure ~\ref{fig:limitcases}~(a) indeed confirms that the stationary distributions of the node2vec random walk for large $\alpha$ and the simple random walk are similar.

Applying $\gamma \to \infty$ to formula~\eqref{eq:evaluecliques}, we obtain
\begin{align}
    \E[\tau^{(C_k)}] = \frac{\alpha+(k-1)(\alpha+(k-2)\beta+2\gamma)}{\alpha+(k-1)\gamma} \xto[\gamma \to \infty]{} 2
\end{align}
which implies that for any node $v$,
\begin{align}
  \pi(v) = \frac{2}{\sum_{k=1}^n 2k n_{C_k}^G} = \frac{1}{n \E[S]}(1+o_{\PP}(1)) = \frac{1}{n\lambda}(1+o_{\PP}(1)),
\end{align}
which is the uniform distribution. The expectation of $2$ is reasonable because in this parameter regime, the node2vec random walk can only make one step inside the clique before leaving the clique.

Figure~\ref{fig:limitcases}~(b) shows that indeed, the distribution for $k\geq 2$ is uniform for the node2vec random walk for large $\gamma$.

Plugging $\alpha = 0, \gamma = 1$ into formula~\eqref{eq:evaluecliques} yields
\begin{align}
    \E[\tau^{(C_k)}] = (k-2)\beta + 2,
\end{align}
which implies that for any node $v$ which is part of a $l$-clique,
\begin{nalign}
    \pi(v) &= \frac{(l-2)\beta + 2}{\sum_{k=1}^n k n_{C_k}^G ((k-2)\beta + 2)} = \frac{(l-2)\beta + 2}{n\E[S((S-2)\beta+2]}(1+o_{\PP}(1)) \nonumber\\
    &= \frac{(l-2)\beta + 2}{\lambda n (\beta(\lambda - 1) +2)}(1+o_{\PP}(1)).
\end{nalign}
For this choice of parameters, the stationary distribution of the node2vec random walk behaves similar to the simple random walk when $\beta$ is close to $\gamma$, while for larger $\beta$, larger cliques are favored.

Figure~\ref{fig:limitcases}~(c) shows that indeed for a larger $\beta$ the stationary distribution of the node2vec random walk deviates more from the one of the simple random walk.

\begin{figure}[tbp]
\centering
    \begin{subfigure}[b]{0.3\textwidth}
    \centering
    \includegraphics[scale=0.55]{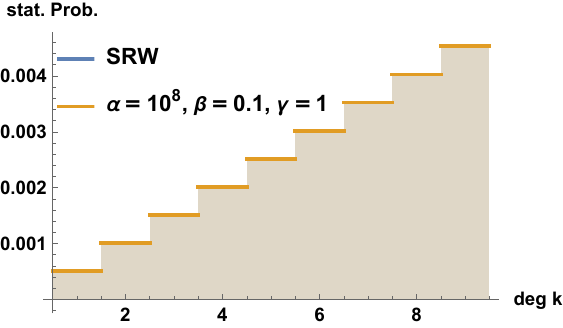}
    \caption{\label{fig:alphatoinfty}}
    \end{subfigure}
\hfill
    \begin{subfigure}[b]{0.3\textwidth}
    \centering
    \includegraphics[scale=0.55]{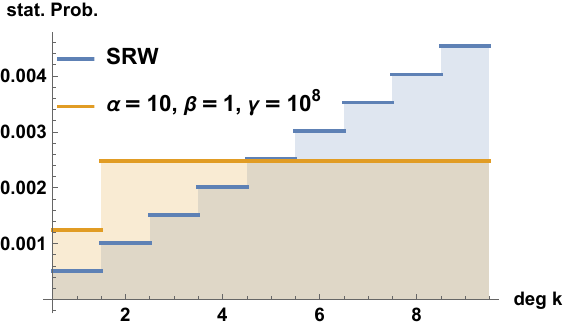}
    \caption{\label{fig:gammatoinfty}}
    \end{subfigure}
\hfill
    \begin{subfigure}[b]{0.3\textwidth}
    \centering
    \includegraphics[scale=0.55]{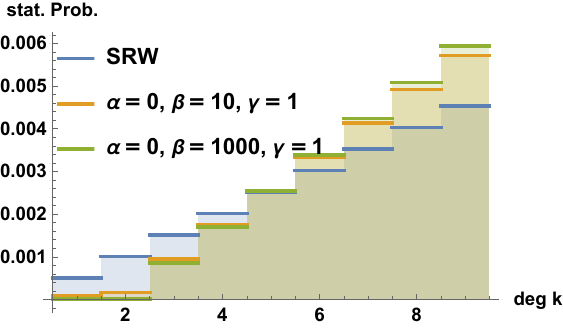}
    \caption{\label{fig:alpha0gamma1}}
    \end{subfigure}
\caption{Stationary distribution of the node2vec random walk and the simple random walk for the limit cases (a) $\alpha\to\infty$, (b) $\gamma\to\infty$ and (c) $\alpha = 0, \gamma = 1$  on a large household model with Poissonian clique sizes.}
\alt{Illustration of 3 bar plots that depict stationary distributions of node2vec random walks for specific limit cases and of the simple random walk for comparison.}
\label{fig:limitcases}
\end{figure}

\subsection{Cliques with an asymmetric number of arms}
The proof of Theorem~\ref{thm:mainthm} strongly relies on the automorphism of the community, and the fact that every community member has the same amount of edges to non-community members. We now show in a small-scale example what the effect is of dropping the latter assumption.

In particular, we prove an expression for the stationary distribution of a triangle with an arbitrary amount of outside arms (an edge which connects a node inside the triangle to a node outside of it) on each node of the triangle. In comparison, the construction of the household model in Definition~\ref{def:simplehousehold} only includes triangles where each node of the triangle has exactly one outside arm. Here we include asymmetric cases. We specifically consider graphs that contain one triangle $u,v,w$, and where each of the nodes $u,v,w$ is connected to an arbitrary number of degree-one nodes ($n,p,m$ respectively).
\begin{figure}[tbp]
    \centering
    \includegraphics[width=0.3\linewidth]{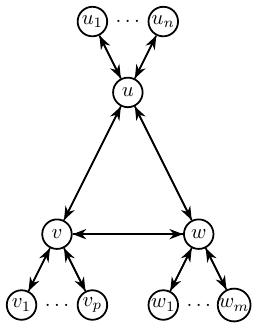}
    \caption{A $3$-clique with an asymmetric number of outside arms.}
    \alt{Illustration of a 3-clique with an arbitrary amount of outside arms on each node.}
    \label{fig:trianglewitharbarms}
\end{figure}

To be more precise, let $n,p,m \in \NN$ and let $G= (V,E)$ be
\begin{align}\label{eq:defmultiarmv}
V \coloneqq \{u,v,w,u_1,\ldots,u_n,v_1,\ldots,v_p,w_1,\ldots,w_m\}
\end{align}
and $E = E_1 \cup E_2 \cup E_3 \cup E_4$ with 
\begin{nalign}\label{eq:defmultiarme}
  E_1 &\coloneqq \{ (x,y)|x,y \in \{ u,v,w \}, x\neq y \}, \\
  E_2 &\coloneqq \{ (u,u_1),\ldots,(u,u_n),(u_1,u),\ldots,(u_n,u) \}, \\
  E_3 &\coloneqq \{ (v,v_1),\ldots,(v,v_p),(v_1,v),\ldots,(v_p,v) \} \text{ and } \\
  E_4 &\coloneqq \{ (w,w_1),\ldots,(w,w_m),(w_1,w),\ldots,(w_m,w) \}. 
\end{nalign}
Note that $E_1,E_2,E_3,E_4$ are pairwise disjoint. Figure~\ref{fig:trianglewitharbarms} visualizes this graph class.

\begin{theorem}
    Let $G = (V,E)$ be the graph defined in~\eqref{eq:defmultiarmv} and~\eqref{eq:defmultiarme}. Then, the stationary probabilities of a node2vec random walk traversing on the directed edges is
    \begin{align}\label{eq:sol_arbarms}
        \pi(e) =
        \begin{cases}
            \tfrac{1}{Z}(\alpha + \beta +n\gamma)(\alpha + \beta +p\gamma)(\alpha + \beta +m\gamma),  e \in E_1, \\
            \tfrac{1}{Z}(\alpha + (n+1)\gamma)(\alpha + \beta +p\gamma)(\alpha + \beta + m\gamma), e \in E_2, \\
            \tfrac{1}{Z}(\alpha + \beta +n\gamma)(\alpha + (p+1)\gamma)(\alpha + \beta +m\gamma), e \in E_3, \\
            \tfrac{1}{Z}(\alpha + \beta +n\gamma)(\alpha + \beta +p\gamma)(\alpha + (m+1)\gamma), e \in E_4,        
        \end{cases}
    \end{align}
    where $Z$ is the normalisation constant.
\end{theorem}

It is surprising that all directed edges inside the triangle have the same stationary distribution, although the number of outside arms is not necessarily equal.

\begin{proof}
    Let $P$ be the transition matrix of the node2vec random walk on the directed edges of $G$. Then, the stationary distribution has to fulfill $\pi = \pi P$ which is equivalent to
    \begin{align}\label{eq:stationaryequations}
        \pi(f) = \sum\limits_{e\in E: e(1)=f(0)} p(e,f)\pi(e) ~~~ \forall f \in \bar{E},
    \end{align}
    where $e(1)$ is the target node of the directed edge $e$ and $f(0)$ is the starting node of the directed edge $f$. Denote the stationary probabilities $\pi_1 = \pi((w,u)),\pi_2 = \pi((u,w)),\pi_3 = \pi((u,v)),\pi_4 = \pi((v,u)),\pi_5 = \pi((v,w)),\pi_6 = \pi((w,v)),\pi_x = \pi((u,u_1)),\pi_y = \pi((w,w_1)),\pi_z = \pi((v,v_1))$.
    Our goal is to solve the system of equations in~\eqref{eq:stationaryequations} to obtain the stationary distribution. By symmetry we know that $\pi((u,u_1)) = \pi((u,u_i))$ for all $i \in \{ 1,\ldots,n \}$ and since $d_{u_1} = 1$, it follows $\pi((u,u_1)) = \pi((u_1,u))$. This implies $\pi(e) = \pi(e')$ for all $e,e'\in E_2$, i.e. the stationary probability is the same for all directed edges in $E_2$. The same argument holds true for $E_3$ and $E_4$. Because of this, we denote the stationary probability for all directed edges in $E_2$ by $\pi_x$, for all in $E_3$ by $\pi_z$ and for all in $E_4$ by $\pi_y$. For directed edges in $E_1$ it is not immediately clear that they have the same stationary probability. We begin by posing the linear equations
    \begin{nalign}
        \pi_1 &= \sum\limits_{e\in E: e(1)=w} p(e,(w,u))\pi(e) \nonumber\\
        &= \pi_2 \frac{\alpha}{\alpha + \beta + m\gamma} + \pi_6 \frac{\beta}{\alpha + \beta + m\gamma} + m\pi_y \frac{\gamma}{\alpha + (m+1)\gamma}
    \end{nalign}
    and 
    \begin{nalign}
        \pi_x &= \sum\limits_{e\in E: e(1)=u} p(e,(u,u_1))\pi(e)\nonumber \\
        &= \pi_x \frac{\alpha}{\alpha + (n+1)\gamma} + \pi_1 \frac{\gamma}{\alpha + \beta + n\gamma} + \pi_4 \frac{\gamma}{\alpha + \beta + n\gamma} + (n-1)\pi_x \frac{\gamma}{\alpha + (n+1)\gamma}.
    \end{nalign}
    The other equations follow analogously. To obtain a probability distribution, we add the equation
    \begin{align}
        \pi_1 + \ldots + \pi_6 + 2n\pi_x +2m\pi_y +2p\pi_z = 1.
    \end{align}
    This means that we can omit one other equation, since $9$ equations are sufficient to solve for $9$ variables, so we omit the equation for z. We obtain the linear system
    $$\begin{pmatrix}[ccccccccc|c] 
    -1 & \frac{\alpha}{\alpha + \beta + m \gamma} & 0 & 0 & 0 & \frac{\beta}{\alpha + \beta + m \gamma} & 0 & \frac{m\gamma}{\alpha + (m+1) \gamma} & 0 & 0 \\ 
    \frac{\alpha}{\alpha + \beta + n \gamma} & -1 & 0 & \frac{\beta}{\alpha + \beta + n \gamma} & 0 & 0 & \frac{n\gamma}{\alpha + (n+1) \gamma} & 0 & 0 & 0 \\
    \frac{\beta}{\alpha + \beta + n \gamma} & 0 & -1 & \frac{\alpha}{\alpha + \beta + n \gamma} & 0 & 0 & \frac{n\gamma}{\alpha + (n+1) \gamma} & 0 & 0 & 0 \\
    0 & 0 & \frac{\alpha}{\alpha + \beta + p \gamma} & -1 & \frac{\beta}{\alpha + \beta + p \gamma} & 0 & 0 & \frac{p\gamma}{\alpha + (p+1) \gamma} & 0 & 0 \\
    0 & \frac{\beta}{\alpha + \beta + m \gamma} & 0 & 0 & -1 & \frac{\alpha}{\alpha + \beta + m \gamma} & 0 & \frac{m\gamma}{\alpha + (m+1) \gamma} & 0 & 0 \\
    0 & 0 & \frac{\beta}{\alpha + \beta + p \gamma} & 0 & \frac{\alpha}{\alpha + \beta + p \gamma} & -1 & 0 & \frac{p\gamma}{\alpha + (p+1) \gamma} & 0 & 0 \\
    \frac{\gamma}{\alpha + \beta + n \gamma} & 0 & 0 & \frac{\gamma}{\alpha + \beta + n \gamma} & 0 & 0 & \frac{-2\gamma}{\alpha + (n+1) \gamma} & 0 & 0 & 0 \\
    0 & \frac{\gamma}{\alpha + \beta + m \gamma} & 0 & 0 & 0 & \frac{\gamma}{\alpha + \beta + m \gamma} & 0 & \frac{-2\gamma}{\alpha + (m+1) \gamma} & 0 & 0 \\
    1 & 1 & 1 & 1 & 1 & 1 & 2n & 2m & 2p & 1 
    \end{pmatrix}$$
    which yields the solutions in~\eqref{eq:sol_arbarms} that fulfill $\pi = \pi P$.
\end{proof}

We believe that similar results should hold for bigger cliques with an arbitrary amount of arms at each node of the clique, only the number of equations that need to be solved will go up significantly.

\section{Proof of Theorem~\ref{thm:mainthm}}\label{sec:proofofmainthm}

Our goal is to obtain the stationary distribution of the node2vec random walk $X=(X_i)_{i\in \NN}$ on a household model $G$. The main difficulty of the proof is that the transition probabilities depend on the number of triangles containing the current and previous node of the random walk, so at each step of the random walk we have to know how many triangles the previously visited and current node are part of. To avoid this problem, we couple the node2vec random walk to a random walk $Y=(Y_i)_{i\in \NN}$ on $G'$ and use the symmetries of the household model. 

The coupling works in the following way. For $v\in V$ denote by $c(v)$ the community $v$ is contained in (it is unique). By construction of the household model, there is exactly one node $v'\in V'$ in the universe graph that corresponds to the community $c(v)$ in the household model. We denote this correspondence by $v'=c(v)$. We define $Y$ by
\begin{align}
    Y_i = c(X_i)
\end{align}
for all $i \in \NN$. This means that $Y$ is a random walk on $G'$ that moves if $X$ moves between communities on $G$ and waits a time step for each jump $X$ makes inside a community. If we change the time scale by ignoring the waiting time and only track the movement, we can calculate the stationary distribution of that walk.

To achieve this, we introduce a different time scale by erasing all the times at which $Y$ is not moving. We define the random walk $Y^* = (Y^*_j)_{j \in \NN}$ recursively via
\begin{nalign}\label{eq:defystar}
    Y^*_1 &\coloneqq Y_{k_1} = Y_1 \\
    Y^*_2 &\coloneqq Y_{k_2}, \text{ where } k_2 = \min\{ l|Y_l \neq Y_1, l > k_1 \} \\
    &\;\;\vdots  \\
    Y^*_j &\coloneqq Y_{k_j}, \text{ where } k_j = \min\{ l|Y_l \neq Y_{k_{j-1}}, l > k_{j-1} \}
\end{nalign}
for all $j \in \NN$, so that $Y^*$ tracks the movement between communities of $X$. 
 
Before deriving the stationary distribution of $Y^*$, we introduce rooted graph automorphisms and automorphic classes. These are inspired by rooted tree isomorphisms (see e.g. \cite{lindeberg_isomorphism_2024}) and will allow us to describe the transitions of the random walk $Y^*$.

\begin{definition}(Rooted graph automorphisms)
    Let $H = (V_H,E_H)$ be a graph with root $r$. A rooted graph automorphism is a bijection $\varphi_r:V_H\to V_H$ that fulfills
    \begin{itemize}
        \item $u,v$ are adjacent in $V_H$ $\Longleftrightarrow$ $\varphi_r(u),\varphi_r(v)$ are adjacent in $V_H$
        \item $\varphi_r(r) = r$.
    \end{itemize}
\end{definition}

Using this definition, we can define rooted automorphic classes which we will use to distinguish classes of nodes in the automorphic communities.

\begin{definition}(Rooted automorphic classes)
    Let $G$ be a household model with corresponding universe graph $G'$ and let $H = (V_H,E_H)$ be an automorphic community in $G$. For $v_1 \in V_H$ define its rooted automorphic class rooted in $v \in V_H$ by
    \begin{nalign}
        aut_{v}(H,v_1) \coloneqq \{v_2 \in V_H\setminus \{ v \}| ~\exists \text{ rooted graph automorphism } \\\varphi_v:V_H\to V_H \text{ such that } \varphi_v(v_1) = v_2 \}.
    \end{nalign}
\end{definition}
Informally, two nodes in the community belong to the same rooted automorphic class rooted in $v$ if they can be swapped while preserving the structure of the graph with respect to the root. 

To simplify notation, from now on we write $v\in H$ instead of $v\in V_H$ for a community $H$. 
Let $K = K_v^H$ be the number of unique rooted automorphic classes of $H$ rooted in $v$ and let $C^1_{v}(H),\ldots,C^K_{v}(H)$ be those unique classes.

\begin{figure}[tbp]
\centering
    \begin{subfigure}[b]{0.3\textwidth}
    \centering
    \includegraphics[scale=0.55]{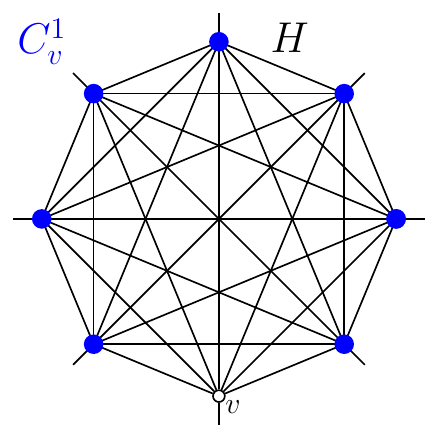}
    \caption{\label{fig:autclassesclique}}
    \end{subfigure}
\hfill
    \begin{subfigure}[b]{0.3\textwidth}
    \centering
    \includegraphics[scale=0.55]{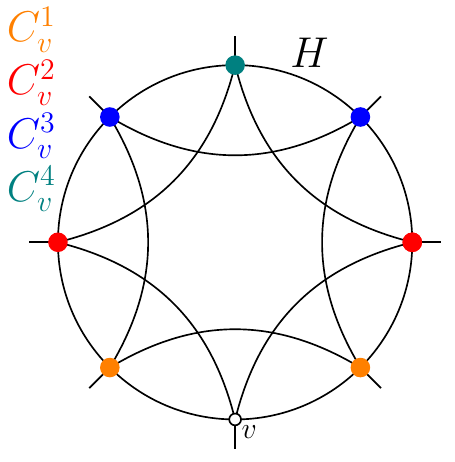}
    \caption{\label{fig:autclassesring}}
    \end{subfigure}
\hfill
    \begin{subfigure}[b]{0.3\textwidth}
    \centering
    \includegraphics[scale=0.55]{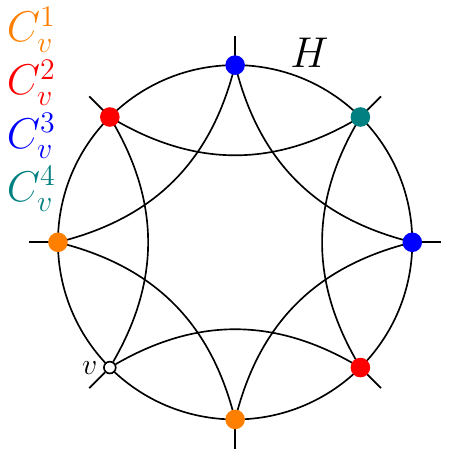}
    \caption{\label{fig:autclassesringrotated}}
    \end{subfigure}
\caption{Rooted automorphic classes of a community $H$ rooted in $v$ where $H$ is (a) an $8-$clique, (b) a ring community of size $8$ and (c) a ring community of size $8$ with a different root.}
\alt{Illustration of the rooted automorphic classes of a clique, a ring community and a ring community with a different root.}
\label{fig:autclasses}
\end{figure}

Now we can partition every community into different rooted automorphic classes given a root $v$.
As an example, Figure \ref{fig:autclasses} (a) shows the rooted automorphic classes of a clique. Since all nodes are automorphically equivalent given the root $v$, they belong to the same rooted automorphic class $C_v^1(H)$. Figure \ref{fig:autclasses} (b) shows the rooted automorphic classes of a different community, where four rooted automorphic classes exist. Although the nodes in $C_v^1(H)$ and $C_v^2(H)$ share a distance of $1$ to the root, they are not automorphically equivalent, leading to their separation into different classes. Figure \ref{fig:autclasses} (c) shows the same community as in (b) with a different root, resulting in different rooted automorphic class assignments compared to (b). Since the communities are automorphic, the root selection does not affect the number of rooted automorphic classes or the node count in each class, but only rotates the classes with respect to the root.

Let $\{ 
u',v' \} \in E'$ with corresponding communities $H_{u'}$ and $H_{v'}$ in $G$. Define the neighboring node of $H_{u'}$ in $H_{v'}$ as
    \begin{align}
        n_{v'}(u') \coloneqq \{ v\in H_{v'}| v \text{ is connected to a node in } H_{u'} \}.
    \end{align}
By the definition of the household model, this node is uniquely defined.
\begin{figure}[tbp]
    \centering
    \includegraphics[width=1\linewidth]{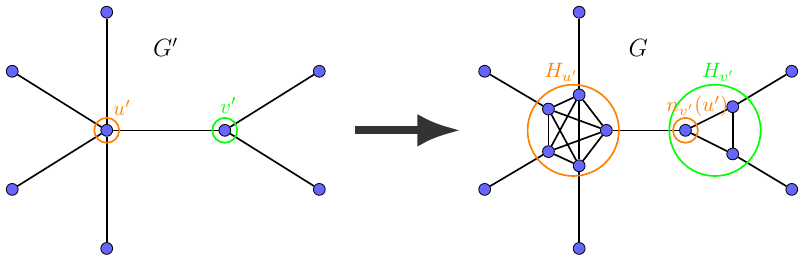}
    \caption{Example of a neighboring node $n_{v'}(u')$ in the household model transformation. $n_{v'}(u')$ is the unique node in $H_{v'}$ that is connected to a node in $H_{u'}$.}
    \label{fig:neighboringnode}
\end{figure}
Figure \ref{fig:neighboringnode} illustrates the definition of neighboring node.

To compute the stationary distribution of $Y^*$, we aim to express the transition probabilities of $Y^*$ in terms of the transition probabilities of the original node2vec random walk $X$. Given that $Y^*$ was previously at position $u'$ and is currently at position $v'$, we distinguish between backtracking—i.e., $Y_{i+1}^* = Y_{i-1}^*$ for some $i \in \mathbb{N}$—and transitions to other nodes, classified by their rooted automorphic class.

For $Y^*$ to backtrack, either $X$ also backtracks or $X$ first moves inside the community before leaving at the node where it entered the community. Denote by $u^*=n_{v'}(u')\in H_{v'}$ the node where $X$ entered $H_{v'}$. The probability to leave again at $u^*$ is given by
\begin{nalign}
    p_\alpha^{H_{v'}} \coloneqq \PP(X \text{ leaves } H_{v'} \text{ at } u^*).
\end{nalign}
We distinguish all other transitions of $Y^*$ by the rooted automorphic class of the exiting node rooted in $u^*$ as
\begin{align}
    p_{u^*,k}^{H_{v'}} &\coloneqq \PP(X \text{ leaves } H_{v'} \text{ at a node } x \in C^k_{u^*}(H_{v'})).
\end{align}
Since the probability to leave $H_{v'}$ is the same for all nodes in that rooted automorphic class, we do not need to distinguish at which node we leave, we only need to know to what rooted automorphic class it belongs and how many nodes are in each class. We denote the size of a rooted automorphic class by
\begin{align}
    a_{u^*,k}^{H_{v'}} \coloneqq |C^k_{u^*}(H_{v'})|
\end{align}
for $k=1,\ldots,K$. We obtain
\begin{nalign}
    P(Y^*_3 = w'|Y^*_2 = v',Y^*_1 = u') = 
        &~~\frac{p_\alpha^{H_{v'}}}{Z} \indic_{\{ w' = u' \}} +
        \frac{p_{u^*,1}^{H_{v'}}}{Z} \indic_{\{ n_{v'}(w') \in  C^1_{u^*}(H_{v'}) \}} \\
        &+\frac{p_{u^*,2}^{H_{v'}}}{Z} \indic_{\{ n_{v'}(w') \in  C^2_{u^*}(H_{v'}) \}} \\ &+
        \ldots +\frac{p_{u^*,K}^{H_{v'}}}{Z} \indic_{\{ n_{v'}(w') \in  C^K_{u^*}(H_{v'}) \}}
\end{nalign}
where $Z = p_\alpha^{H_{v'}} + a_{u^*,1}^{H_{v'}}p_{u^*,1}^{H_{v'}} + a_{u^*,2}^{H_{v'}}p_{u^*,2}^{H_{v'}} + \ldots + a_{u^*,K}^{H_{v'}}p_{u^*,K}^{H_{v'}}$. 
Similarly to node2vec, $Y^*$ is a Markov chain on the directed edges. 
The next theorem shows that the stationary distribution of $Y^*$ is uniform.

\begin{theorem}\label{thm:statdistrofystar}
    Let $Y^*$ be the random walk defined in \eqref{eq:defystar} on a universe graph $G'$ corresponding to the household model $G$. Then we have
    \begin{align}
        \pi^{Y^*}(v') = \frac{d_{v'}}{2|E'|}
    \end{align}
    for all $v' \in V'$.
\end{theorem}

\begin{proof}
    Denote $u^* = n_{v'}(u')$. For $\{ u',v' \} \in E'$ we have 
    \begin{align}
        \sum\limits_{w' \in V'} \indic_{\{ n_{v'}(w') \in  C^k_{u^*}(H_{v'}) \}} = |\{ w'\in V':  n_{v'}(w')\in  C^k_{u^*}(H_{v'}) \}| = |C^k_{u^*}(H_{v'})| = a_{u^*,k}^{H_{v'}}
    \end{align}
    for all $k=1,\ldots,K$.
    Using this and the fact that $P$ is a probability measure, for $u',v'\in V'$ we obtain
    \begin{nalign}\label{eq:pprobmeasure}
        1 = &~~\sum\limits_{w' \in V'} P(Y^*_3 = w'|Y^*_2 = v',Y^*_1 = u') \\
        = &~~\sum\limits_{w' \in V'}\frac{p_\alpha^{H_{v'}}}{Z} \indic_{\{ w' = u' \}} + \sum\limits_{w' \in V'} \frac{p_{u^*,1}^{H_{v'}}}{Z} \indic_{\{ n_{v'}(w') \in  C^1_{u^*}(H_{v'}) \}} \\
        &+\sum\limits_{w' \in V'}\frac{p_{u^*,2}^{H_{v'}}}{Z} \indic_{\{ n_{v'}(w') \in  C^2_{u^*}(H_{v'}) \}} \\ 
        &+\ldots +\sum\limits_{w' \in V'}\frac{p_{u^*,K}^{H_{v'}}}{Z} \indic_{\{ n_{v'}(w') \in  C^K_{u^*}(H_{v'}) \}} \\
        = &~~\frac{p_\alpha^{H_{v'}}}{Z}  + a_{u^*,1}^{H_{v'}}\frac{p_{u^*,1}^{H_{v'}}}{Z} +a_{u^*,2}^{H_{v'}}\frac{p_{u^*,2}^{H_{v'}}}{Z} +\ldots +a_{u^*,K}^{H_{v'}}\frac{p_{u^*,K}^{H_{v'}}}{Z} .
    \end{nalign}
    Since $H_{v'}$ is an automorphic community, the number and types of rooted automorphic classes remain the same for each root (see the example in Figure \ref{fig:autclasses} (b) and (c)), although the class labels may be permuted. Thus, for $\{ v',w'\} \in E'$,
    \begin{align}\label{eq:automorphicidentity}
       \sum\limits_{u' \in V'} \indic_{\{ n_{v'}(w') \in  C^k_{u^*}(H_{v'}) \}} = |\{ u'\in V':  n_{v'}(w')\in  C^k_{u^*}(H_{v'}) \}| =  a_{u^*,k}^{H_{v'}}
    \end{align}
    for all $k=1,\ldots,K$. Combining \eqref{eq:automorphicidentity} and \eqref{eq:pprobmeasure} yields
    \begin{nalign}
        &~~\sum\limits_{u' \in V'} P(Y^*_3 = w'|Y^*_2 = v',Y^*_1 = u') \\
        = &~~\sum\limits_{u' \in V'}\frac{p_\alpha^{H_{v'}}}{Z} \indic_{\{ w' = u' \}} + \sum\limits_{u' \in V'}
        \frac{p_{u^*,1}^{H_{v'}}}{Z} \indic_{\{ n_{v'}(w') \in  C^1_{u^*}(H_{v'}) \}} \\
        &+\sum\limits_{u' \in V'}\frac{p_{u^*,2}^{H_{v'}}}{Z} \indic_{\{ n_{v'}(w') \in  C^2_{u^*}(H_{v'}) \}} \\ &+
        \ldots +\sum\limits_{u' \in V'}\frac{p_{u^*,K}^{H_{v'}}}{Z} \indic_{\{ n_{v'}(w') \in  C^K_{u^*}(H_{v'}) \}}  \\
        = &~~\frac{p_\alpha^{H_{v'}}}{Z}  +
        a_{u^*,1}^{H_{v'}}\frac{p_{u^*,1}^{H_{v'}}}{Z} +a_{u^*,2}^{H_{v'}}\frac{p_{u^*,2}^{H_{v'}}}{Z} +
        \ldots +a_{u^*,K}^{H_{v'}}\frac{p_{u^*,K}^{H_{v'}}}{Z}  \\
        = &~~1.
    \end{nalign}
    This implies that the transition probability matrix of $Y^*$ (seen as a process on $\bar{E'}$) is bistochastic and as in the proof of Theorem 3.2. in \cite{meng_analysis_2020} has left eigenvector $(1,\ldots,1)$. With the uniqueness of the Perron-Frobenius eigenvector, this implies that
    \begin{align}
        \bar{\pi}^{Y^*}(e) = \frac{1}{|\bar{E}'|}.
    \end{align}
    With the projection \eqref{eq:projectionstatdistredgestonodes} we obtain
    \begin{align}
        \pi^{Y^*}(v') = \sum_{e\in\bar{E}': e(1)=v'} \bar{\pi}^{Y^*}(e) = \sum_{e\in\bar{E}': e(1)=v'} \frac{1}{|\bar{E}'|} = \frac{d_{v'}}{2|E'|}
    \end{align}
\end{proof}

By using Theorem \ref{thm:statdistrofystar} and by relating the different time scales to each other, we can obtain the stationary distribution of the node2vec random walk on $G$.
To compare the different time scales, we introduce some notation.
The time $X$ spends in node $v \in V$ in $t$ steps is denoted by
\begin{align}
    t_v = \sum\limits_{i=1}^t \indic_{\{ X_i = v\}},
\end{align}
the time $Y^*$ spends in the node $v' \in V'$ in $T$ steps by
\begin{align}
    T_{v'} = \sum\limits_{j=1}^T \indic_{\{ Y^*_j = v'\}},
\end{align}
and the time $X$ spends in the community $v' \in V$ in $t$ steps by
\begin{align}
    t_{v'} = \sum\limits_{i=1}^t \indic_{\{ X_i \in v'\}}.
\end{align}

For our main theorem we need a law of large numbers for a random number of independent random variables.

\begin{lemma}(Generalized LLN)[follows from Theorem 10.1 in \cite{revesz_laws_1968}]\label{lem:llnrandom}\\
    Let $\xi_1,\xi_2,\ldots$ be a sequence of i.i.d. random variables with $\E[\xi_1] < \infty$ and $\nu_1,\nu_2,\ldots$ be a sequence of positive integer valued random variables which fulfill $\nu_n \xto[n \to \infty]{a.s.}\infty$. Then,
    \begin{align}
        \frac{1}{\nu_n}\sum\limits_{k=1}^{\nu_n} \xi_k \xto[n \to \infty]{a.s.} \E[\xi_1].
    \end{align}
\end{lemma}

We then have all tools to prove Theorem~\ref{thm:mainthm}.

\begin{proof}[Proof of Theorem~\ref{thm:mainthm}]
    Let $\bar{X_i}$ be the to $X_i$ corresponding Markov chain on the directed edges $\bar{E}$. Proposition 73 of \cite{serfozo_basics_2009} states that on $\bar{E}$ we have
    \begin{align}
        \lim\limits_{t\to\infty} \frac{1}{t}\sum\limits_{i=1}^t \indic_{\{\bar{X_i}=e\}} = \bar{\pi}(e)
    \end{align}
    where $\bar{\pi}(e)$ is the stationary distribution of $\bar{X_i}$ on $\bar{E}$. Together with the projection~\eqref{eq:projectionstatdistredgestonodes} between the stationary distribution on the directed edges and on the nodes, this implies
    \begin{align}
        \pi(v) = \sum\limits_{e:e(1) = v} \bar{\pi}(e) = \lim\limits_{t\to\infty} \frac{1}{t} \sum\limits_{e:e(1) = v}  \sum\limits_{i=1}^t \indic_{\{\bar{X_i}=e\}} = \lim\limits_{t\to\infty} \frac{t_v}{t}.
    \end{align}
    Using this, we begin by rewriting the above fraction as
    \begin{align}\label{eq:separationoftimes}
        \frac{t_v}{t} = \frac{T_{v'}}{T} \frac{T}{t} \frac{t_v}{t_{v'}} \frac{t_{v'}}{T_{v'}}.
    \end{align}
    In the following, we compute the limits of the four terms on the right-hand side of \eqref{eq:separationoftimes} to obtain an expression for $\pi(v)$. By Theorem \ref{thm:statdistrofystar},
    \begin{align}\label{eq:convergenceofYstar}
        \frac{T_{v'}}{T} \xto[T \to \infty]{\PP} \frac{d_{v'}}{2\lvert E'\rvert}.
    \end{align}
    Since by assumption all nodes in a community are automorphically equivalent, in the limit the time spent in a community gets equally divided between all nodes of that community. We obtain
    \begin{align}
        \frac{t_v}{t_{v'}} \xto[t \to \infty]{\PP} \frac{1}{d_{v'}}.
    \end{align}
    For the next term, we define $\tau_i$ as the time $X$ spends in the i-th visited community and $\tau_i^{(H)}$ as the time $X$ spends in the i-th visited community of type $H\in \cH_G$. If $Y^*$ jumps $T$ times, then $X$ jumps $\sum\limits_{i=1}^T \tau_i$ times. This leads to
    \begin{align}
        \frac{T}{t} = \frac{T}{\sum\limits_{i=1}^T \tau_i} = \frac{1}{\frac{1}{T}\sum\limits_{i=1}^T \tau_i}.
    \end{align}
    We define the amount of visited nodes which correspond to a community of type $H\in \cH_G$ during $T$ steps of $Y^*$ as
    \begin{align}
        T_H \coloneqq \sum\limits_{\substack{w'\in V': \\ type(w')=H}} T_{w'}.
    \end{align}
    As the number of nodes of $G$ is $|V|=n$, we can rearrange the sum and group the $\tau_i$ by their community type to obtain
    \begin{align}
        \frac{1}{T}\sum\limits_{i=1}^T \tau_i = \frac{1}{T}\sum\limits_{H\in \cH_G} \sum\limits_{j=1}^{T_H} \tau_{j}^{(H)} = \sum\limits_{H\in \cH_G} \frac{T_H}{T}\frac{1}{T_H}\sum\limits_{j=1}^{T_H} \tau_{j}^{(H)}.
    \end{align}
    For a fixed $H$, the times $\tau_{j}^{(H)}$ are identically distributed and independent. Additionally, since the underlying graph is connected and the node2vec random walk is recurrent for $\alpha, \beta, \gamma > 0$, they have finite expectation for all $j$. For fixed $H$ the times $T_H$ are positive integer valued random variables that fulfill $T_H \xto[T \to \infty]{a.s.} \infty$. This means we can use Lemma~\ref{lem:llnrandom} to obtain
    \begin{align}
        \frac{1}{T_H}\sum\limits_{j=1}^{T_H} \tau_{j}^{(H)} \xto[T \to \infty]{a.s.} \E[\tau_1^{(H)}]
    \end{align}
    for all $H\in \cH_G$. Additionally, using~\eqref{eq:convergenceofYstar} we have
    \begin{align}
        \frac{T_H}{T} = \sum\limits_{\substack{w'\in V': \\ type(w')=H}} \frac{T_{w'}}{T} \xto[T \to \infty]{\PP} \sum\limits_{\substack{w'\in V': \\ type(w')=H}} \frac{d_{w'}}{2|E'|} = \frac{|H| n_H^{G}}{2|E'|}
    \end{align}
    for all $k \in \{ 1,\ldots,n \}$. We can conclude that
    \begin{align}
        \frac{T}{t} \xto[T \to \infty]{\PP} \frac{2|E'|}{\sum\limits_{H \in \cH_G} |H| n_H^{G}\E[\tau_1^{(H)}]}.
    \end{align}
    Using Lemma~\ref{lem:llnrandom} again, we obtain
    \begin{align}
        \frac{t_{v'}}{T_{v'}} = \frac{\sum\limits_{i=1}^{T_{v'}} \tau_{i}^{(\Tilde{H})}}{T_{v'}} = \frac{1}{T_{v'}}\sum\limits_{i=1}^{T_{v'}} \tau_{i}^{(\Tilde{H})} \xto[T_{v'} \to \infty]{\PP} \E[\tau_1^{(\Tilde{H})}].
    \end{align}
    Since the random walks are recurrent, $T \to \infty$ implies $T_{v'} \to \infty$ and $T_k \to \infty$. By construction $T \to \infty$ also implies $t \to \infty$ and we let $T \to \infty$  on both sides of~\eqref{eq:separationoftimes} to obtain
    \begin{align}
        \pi(v) = \frac{d_{v'}}{2\lvert E' \rvert}\frac{1}{d_{v'}}\frac{2|E'|}{\sum\limits_{H \in \cH_G} |H| n_H^{G}\E[\tau_1^{(H)}]}\E[\tau_1^{(\Tilde{H})}] = \frac{\E[\tau_1^{(\Tilde{H})}]}{\sum\limits_{H \in \cH_G} |H| n_H^{G}\E[\tau_1^{(H)}]}.
    \end{align}
    Since all $\tau_i^{(H)}$ are i.i.d., we can use $\E[\tau_1^{(H)}] = \E[\tau^{(H)}]$ to get
    \begin{align}
        \pi(v) = \frac{\E[\tau^{(\Tilde{H})}]}{\sum\limits_{H \in \cH_G} |H| n_H^{G}\E[\tau^{(H)}]}.
    \end{align}
\end{proof}

\section{Conclusion and discussion}
In this paper, we have investigated the stationary distribution of the popular node2vec random walk on a specific class of graph models with a community structure. We have obtained exact results for the stationary distribution of the node2vec random walk for general walk parameters. 

We have demonstrated that the parameters governing non-backtracking behavior and the local and global exploration significantly influence the stationary distribution. Adjusting these parameters can shift the random walk from being uniformly distributed on nodes or edges to being biased towards larger communities. While node2vec random walks have been popular due to the tuneable parameters in practice, this mathematically demonstrates the power of these node2vec random walks to perform different types of random walks.

In this paper, we have only investigated the behavior of the stationary distribution in so-called household models, as these models allow to couple the node2vec random walk, which is heavily dependent on the triangle structure of the graph, to a known random walk on a `universe graph' which is free of the original triangle structure. We show that under slight modifications of this framework it is still possible to obtain the stationary distribution of the walk, but investigating the node2vec random walk on other types of graph models such as random regular graphs or stochastic block models would be an interesting point for further research. To achieve this, different methods would need to be employed, as the current transformation does not work in that setting.

Secondly, the stationary distribution is only one statistic of a random walk. Specifically, the node2vec random walks that are typically used in practice are rather short. Therefore, it would be interesting to quantify the behavior of other random walk statistics, such as mixing times or hitting times.

Finally, it would be interesting to investigate how the behavior of the node2vec random walk impacts algorithms that use it. In~\cite{barot_community_2021}, it was shown that adding the non-backtracking parameter of the node2vec random walk improves community detection. It would be interesting to see whether the local vs global exploration parameters have a similar effect, and we believe that our analysis of the stationary distribution provides the first step in performing such an analysis. 

\paragraph{Acknowledgments.}
The authors thank Gianmarco Bet and Luca Avena for fruitful discussions.

\printbibliography
\end{document}